\documentclass{article}
\usepackage{amsmath}
\usepackage{amssymb}
\usepackage{enumerate}
\title{On approximation of functions satisfying defective renewal equations}

\author{{\sc C. Sang\"{u}esa}}

\date{}

\begin{document}
\baselineskip20pt
\begin{titlepage}
\setcounter{page}{1} \maketitle

\bigskip \bigskip
\begin{abstract}Functions satisfying a defective renewal equation arise commonly in applied probability models.  Usually these functions don't admit a explicit expression.  In this work we consider to approximate them by means of a gamma-type operator given in terms of the Laplace transform of the initial function.  We investigate  which conditions on the initial parameters of the renewal equation give optimal order of uniform convergence in the approximation.  We apply our results to ruin probability in the classical risk model, paying special attention to mixtures of gamma claim amounts.
\end{abstract}

\bigskip \bigskip
2000 Mathematics subject classification: Primary: 60K05, 60E10, 41A25; Secondary:  41A35, 91B30

\bigskip \bigskip
{\it Key words and phrases}: defective renewal equation, Laplace transform, rate of convergence, ruin probability, gamma distribution
\vfill \eject \noindent Short title: Approximations in defective renewal equations
\bigskip \noindent

\bigskip \bigskip \noindent
Mailing address:

\bigskip \noindent
C. Sang\"uesa

\smallskip \noindent
Departamento de M\'etodos Estad\'{\i}sticos,

\smallskip \noindent
Facultad de Ciencias, Universidad de Zaragoza,

\smallskip \noindent
Pedro Cerbuna, 12,

\smallskip \noindent
50009 Zaragoza (Spain).

\smallskip \noindent
csangues@unizar.es \vfill\eject

\end{titlepage}

\newtheorem{theorem}{\bf Theorem}

\newtheorem{remark}{\bf Remark}

\newtheorem{corollary}{\bf Corollary}

\newtheorem{lemma}{\bf Lemma}
\newtheorem{example}{\bf Example}
\newenvironment{proof}{\it Proof. \rm}{$\quad \Box$}
\newtheorem{proposition}{\bf Proposition}
\renewcommand{\thetheorem}{\thesection.\arabic{theorem}}
\renewcommand{\thelemma}{\thesection.\arabic{lemma}}
\renewcommand{\theproposition}{\thesection.\arabic{proposition}}
\renewcommand{\thecorollary}{\thesection.\arabic{corollary}}
\renewcommand{\theremark}{\thesection.\arabic{remark}}
\renewcommand{\theexample}{\thesection.\arabic{example}}
\newcommand{\discre}[2]{#1^{{\scriptscriptstyle
 \bullet} #2}}

\newcommand{\nst}{\mathbb{N}^{*}}
\newcommand{\Nn}{\mathbb{N}}
\newcommand{\operint}[2]{#1\left(\frac{S([t#2]+1)}{t}\right)}
\newcommand{\rmref}[1]{{\rm (\ref{#1})}}
\newcommand{\rtref}[1]{{\rm \ref{#1}}}
\newcommand{\Rr}{\hbox{\rm I\kern-2.25pt R}}
\newcommand{\seq}[1]{(#1_{i})_{i\in \nst}}
\newcommand{\seqn}[1]{(#1_{i})_{i=1}^{n}}

\section{Introduction}

For a given interval $I\in \Rr$, let $C(I)$ be the class of continuous functions $g:I\rightarrow \Rr$. The aim of this paper is to study the approximation of a function
$g \in C([0,\infty))$ in terms of its Laplace transform.  To this end, we assume that $|g(u)|=O(e^{\gamma u})$ as $u\rightarrow \infty$, for some $\gamma\geq 0$.  Then, the Laplace transform of $g$
\begin{equation}
\tilde{g}(t):=\int_0^\infty{e^{-tu}g(u)du}\, , \quad t>\gamma
\label{dflapt}
\end{equation}
is well defined and infinitely differentiable.  Moreover, we can approximate $g$ in terms of the derivatives of its Laplace transform. From now on, for a given function $g$, $g^{(n)}$ will denote its n-th derivative ($g^{(0)}:=g$).  We define the following operator \begin{equation}
L_{t}^{*}g(u) =\frac{(-t)^{[tu]+1} }{ \Gamma([tu]+1)}\tilde{g}^{([tu])}(t)
\, ,\quad u\geq 0,\  t>\gamma\label{delets}
\end{equation}
where $[u]$ indicates the largest
integer less than or equal to $u$ and $\Gamma(\cdotp)$ is the gamma function.  The approximation properties of (\ref{delets}), as $t\rightarrow \infty$, can be studied taking into account that the previous formula admits the following representation (cf. \cite[example (c), p. 92]{addire}).  Let $(S(t),\ t\geq0)$ be a collection of random variables such that $S(0)=0$ and, for each $t>0$, $S(t)$ is a gamma $\Gamma(t,1)$ random variable.  Recall that a gamma $\Gamma(\alpha,\beta)$ random variable has density given by \begin{equation}f_{\alpha,\beta}(u):=\frac{1}{\Gamma(\alpha)}\beta^{\alpha} u^{\alpha-1}e^{-\beta u},   \quad u\geq 0\quad \beta>0, \  \alpha> 0.\label{gamden} \end{equation}  By differentiation under the integral sign in (\ref{dflapt}) it can be seen that
\begin{equation}
L_{t}^{*}g(u)
=Eg\left(\frac{S([tu]+1)}{t}\right),\quad u\geq 0,\  t>\gamma
\label{repgam}\end{equation}
A modification of the operator defined in (\ref{repgam}) was used in \cite{saunif} to approximate the distribution function $F_{X}$ of a nonnegative random variable $X$ by means of its Laplace-Stieltjes transform, using that (\ref{repgam}), when $g=F_{X}$ can be rewritten in terms of a well-known inversion formula for this transform (in the beginning of Section 3.2. we explain this connection).  By considering the inversion formula (\ref{delets}) using the Laplace transform instead of the Laplace-Stieltjes transform, we can therefore widen the class of functions under consideration, using, at the same time the general convergence results given in  \cite{saunif}.  The connection of both inversion formulas through the same operator was considered also in \cite{addire}. Also, it is interesting to point out that similar approximation formulas involving the Laplace transform, have been used in the literature in order to obtain results concerning characterizations of life distributions in reliability and shape properties of renewal functions (see \cite{kelapl} and the references therein).  Moreover, in a recent paper (cf. \cite{matenu}), we can find interesting numerical comparisons for different inversion formulas involving the Laplace-Stieltjes transform of measures concentrated on the positive semiaxis. It should be mentioned that (\ref{repgam}), applied to distribution functions, is the so-called Widder's formula in  \cite{matenu} (in Section 3.2. we give more details).

\begin{remark}The well-known Post Widder's inversion formula for Laplace transforms, that is,
\begin{equation}W_{t}g(u) =\frac{(-1)^{t-1} }{(t-1)!}\left(\frac{t}{u}\right)^{n}\tilde{g}^{(t-1)}(t/u),\quad u>0,\ t\in \Nn\label{postwi}\end{equation}
admits a similar probabilistic interpretation (cf. Feller {\rm \cite[p.233]{feanin}}), as we have
\begin{equation}W_{t}g(u) =Eg\left(\frac{uS(t)}{t}\right)\,\quad u>0,\ t\in \Nn\label{wirepr}\end{equation}
Observe the main differences of the inversion formula \rmref{delets} with Post-Widder inversion. In the first one, for $t$ fixed, the order of differentiation increases with $u$, whereas the point at which the Laplace transform is applied remains fixed.  In Post-Widder inversion, for $t$ fixed, the order of differentiation is fixed, whereas the point at which the Laplace transform is applied varies with $u$. \label{powide}\end{remark}

%

Our aim in this paper is to construct, as in \cite{saunif} an accelerated approximation to (\ref{delets}) and apply it to functions satisfying a defective renewal equation of the form:
\begin{equation}
m(u)=\phi\int_{0}^{u}m(u-y)dF(y)dy+v(u),\quad u\geq0, \label{renewa}\end{equation}
in which $F$ is the distribution function of a nonnegative random variable with $F(0)=0$, $\phi$ is a parameter such that $0<\phi<1$ and $v:[0,\infty ) \rightarrow \Rr$ is a locally bounded function. Many quantities of interest in applied probability (ruin probability in insurance theory for instance) satisfy a defective renewal equation. For specific references, along with the properties we are going to mention below, see for instance \cite[p.152]{wililu}.  It is known that there is a unique locally bounded solution of (\ref{renewa}).  In fact, let $F^{*n}$ the n-th convolution of $F$ with itself ($F^{0}$ being the point mass at 0). Define $G:=\sum_{n=0}^{\infty}(1-\phi)\phi^{n}F^{*n}$.  Then, the above mentioned solution to (\ref{renewa}) is given by
\begin{equation}m(u)=\frac{1}{1-\phi}\int_{(0,u]}v(u-y)dG(y)+v(u),\quad u\geq 0. \label{rensol} \end{equation} Only in very specific situations one can find an explicit solution to $m$ using (\ref{rensol}).  From now on, we will assume that $F$ is absolutely continuous, having density $f$, and therefore, (\ref{renewa}) becomes
\begin{equation}
m(u)=\phi\int_{0}^{u}m(u-y)f(y)dy+v(u),\quad u\geq0, \label{renewc}\end{equation}
Note that in this case the Laplace transform of $m$ can be written, using (\ref{renewc}) as
\begin{equation}\tilde{m}(t)=\frac{{\tilde v}(t)}{1-\phi{\tilde f}(t)}\label{lapren}\end{equation}
and we can build the approximation $L_{t}^{*}m$ defined in (\ref{delets}), in cases that $m$ cannot be computed in an explicit way.

The paper is organized as follows.  In the next Section we introduce a modification of the operator defined in (\ref{delets}) improving the rate of approximation and apply it to renewal functions, investigating conditions under which the rate of uniform convergence is optimal.  In Section 3 we consider a particular application of or our results in the context of ruin theory in insurance risk models.
\section{The accelerated approximation.  Application to renewal functions}
In order to improve the rate of approximation given by (\ref{delets}), we will consider, in a similar way as in \cite{saunif}  the following accelerated approximation for a given function $g\in C([0,\infty))$
\begin{equation}M_{t}^{[2]}g\left(\frac{k}{t}\right)=
    \left\{%
\begin{array}{ll}
   g(0), & \hbox{if } k=0; \\\displaystyle

2L_{2t}^{*}g\left(\frac{2k-1}{2t}\right)-L_{t}^{*}g\left(\frac{k-1}{t}\right)
 & \hbox{if }k=1,2,\dots \\
\end{array}%
\right.\label{dftile}\end{equation}
and for $u>0$ such that $u\neq k/t,\ k=1,2,\dots$,
\begin{equation}M_{t}^{[2]}g(u)=(tu-[tu])M_{t}^{[2]}g\kern-2pt\left(\frac{[tu]+1}{t}\kern-2pt\right)+([tu]+1-tu) M_{t}^{[2]}g\kern-2pt\left(\frac{[tu]
}{t}\kern-2pt\right).\label{dftilt}\end{equation}
$M_{t}^{[2]}g$ improves the initial order of convergence of $L_{t}^{*}g$ (at most $1/t$) to $1/t^{2}$, for suitable functions. In particular, in \cite{saunif} a class of functions was studied under which this order of convergence holds uniformly.  We now introduce this class of functions.

From now on we will denote by $C^{n}([0,\infty))$ the subclass of functions in  $C([0,\infty))$ having continuous
n-th derivative on $(0,\infty)$. Also, for a given function $g:I\rightarrow \Rr$, $\|g\|$ will denote its uniform norm, that is
\[\|g\|:=\sup_{u\in I} |g(u)|.\]
For a given subset $A\subset I$, we will use the notation $\|g\|_{A}:=\sup_{u\in A} |g(u)|$.  We introduce the following class of functions:

\begin{equation}
{\cal D}_{1}:=\{g\in C^{4}([0,\infty)): \quad \|g''(u)\|< \infty
\ \hbox{and} \
\|u^{2}g^{iv}(u)\|<\infty\}.\label{dfcla2}\end{equation}
\begin{remark} For $g\in {\cal D}_{1}$ we have that $\|u g'''(u)\|\leq \|u^{2}g^{iv}(u)\|<\infty$ (cf. {\rm \cite[p.571]{saunif}}).  The quantity $\|u g'''(u)\|$ will also appear in our error bounds.\end{remark}

The following result will play an important role from now on.

\begin{theorem}[\cite{saunif}, p.571] Let $g\in {\cal D}_{1}$, with ${\cal D}_{1}$ as
defined in \rmref{dfcla2} and let $M^{[2]}_{t}g, \ t>0$ be as defined
in \rmref{dftile}-\rmref{dftilt}.  We have
\[\|M^{[2]}_{t}g-g\|\leq \frac{1}{8t^{2}}\|g''(u)\|
+\frac{1}{6t^{2}}\|ug'''(u)\|+\frac{9}{16t^{2}}\|u^{2}g^{iv}(u)\|<\infty.\]
\label{tetota}\end{theorem}

\begin{remark} It would be interesting to compare the accelerated inversion formula $M_{t}^{[2]}g$ with a similar procedure for the Post-Widder inversion formula (\ref{postwi}), given by \[G_{t}^{[2]} g(u) := 2W_{2t}g(u) -W_{t}g(u).\] Previous expression is the classical Stehfest enhancement of order two for the Post–Widder formula. A numerical comparison example between both accelerated inversion formulas was given in {\rm \cite[Example 2.1. p. 564]{saunif}}. The test function considered was $g(u)=1-(1-p)e^{-pu}, u\geq 0$,  and the parameters taken were $p=0.1$ and $t=5$. The advantage of this function is that both approximations can be easily computed.  The numerical performance of each method for different values of $u$ can be seen in {\rm \cite[Table 2, p. 567]{saunif}}, being $M_{t}^{[2]}g$ more accurate, specially for big values of $u$. Roughly speaking, the better accuracy of $M_{t}^{[2]}g(u)$ can be explained by the fact that the variability of the  underlying random variable defining $L_{t}^{*}g(u)$ (recall \rmref{repgam})  has, for fixed $t$ and $u\rightarrow \infty$, less variability than the one defining the Post-Widder operator (recall \rmref{wirepr}). On the other hand, the use of $M^{[2]}_{t}g$ instead of $L_{t}^{*}g$, improves substantially the numerical performance of the approximation for values $u$ close to the origin (as noticed in {\rm \cite{matenu}}, Widder's formula has less precision for values $u$ close to the origin).  See {\rm \cite[Example 2.1. p. 564]{saunif}} for a more detailed discussion. \label{renume}\end{remark}

Our aim is to consider the renewal function given in (\ref{renewc}) in order to obtain conditions on $F$ and $v$ such that $m\in {\cal D}_{1}$ , with ${\cal D}_{1}$ as defined (\ref{dfcla2}).  In this case,  by Theorem \ref{tetota}, $M_{t}^{[2]}m$, as defined in (\ref{dftile})-(\ref{dftilt}) has order of convergence $1/t^{2}$.  To this end, we need suitable expressions for the derivatives of $m$.  From now on, we will denote by $C^{n}_{0}([0,\infty))$ to the subclass of functions in $C^{n}([0,\infty))$ such that $\displaystyle \lim_{t\downarrow 0}g^{(k)}(t)$ exists for all $k=0,1,\dots n$ and is finite.

First of all we state a technical lemma in order to justify differentiation under the integral sign in expressions similar to (\ref{renewc}).  This result will be systematically used along the paper.
\begin{lemma} Let $c:(0,\infty)\rightarrow \Rr$ be a function satisfying
\[c(u)=\int_{0}^{u}a(u-z)b(z)dz,\quad u>0,\]
in which $a\in C^{1}_{0}([0,\infty))$ and $b\in C((0,\infty))$  is such that $\int_{0}^{u}|b(z)|dz<\infty$ for all $u>0$.  Then $c$ is differentiable for all $u>0$ and verifies
\[c'(u)=\int_{0}^{u}a'(u-z)b(z)dz+a(0)b(u)\]\label{letech}
\end{lemma}
\begin{proof}
Consider $u>0$, let $0<\delta<u$ and let $0<|h|\leq \delta$. To show the assertion we will check that
\[\lim_{h\rightarrow 0}  \frac{1}{h}\left(\int_{0}^{u+h}a(u+h-z)b(z)dz-\int_{0}^{u}a(u-z)b(z)dz\right)=\int_{0}^{u}a'(u-z)b(z)dz+a(0)b(u)\]
Define $a(u)=a'(0)u+a(0)$ if $-\delta\leq u<0$ (in this way  $a(u)$ is continuously differentiable on $[-\delta,\infty)$). Observe firstly that we can write
\begin{align}&\frac{1}{h}\left(\int_{0}^{u+h}a(u+h-z)b(z)dz-\int_{0}^{u}a(u-z)b(z)dz\right)\nonumber\\=&\int_{0}^{u}\frac{1}{h}\left(a(u+h-z)-a(u-z)\right)b(z)dz
+\frac{1}{h}\int_{u}^{u+h}a(u+h-z)b(z)dz\label{princi}\end{align}
To deal with the first term, note that when $h<0$ and $z$ varies between 0 and u, then $u+h-z$ varies between $h$ and $u+h$, thus taking negative values.  That is why we defined $a$ on $[-\delta,\infty)$. Notwithstanding, as $a$ is differentiable on this interval, and $-\delta<h$, we have always $|h^{-1}(a(u+h-z)-a(u-z))b(z)|\leq \|a'\|_{[-\delta,u+\delta]}|b(z)|$. As the right-hand side is an integrable function, we can apply dominated convergence theorem to write
\begin{equation}\lim_{h\rightarrow 0}\int_{0}^{u}\frac{1}{h}\left(a(u+h-z)-a(u-z)\right)b(z)dz=\int_{0}^{u}a'(u-z)b(z)dz\label{prime}\end{equation}
As for the second term in (\ref{princi}), we can write
\[\frac{1}{h}\int_{u}^{u+h}a(u+h-z)b(z)dz=\frac{1}{h}\int_{u}^{u+h}(a(u+h-z)-a(u-z))b(z)dz+\frac{1}{h}\int_{u}^{u+h}a(u-z)b(z)dz\]
 To deal with the first term we see that  $|h^{-1}(a(u+h-z)-a(u-z))b(z)|\leq \|a'\|_{[-\delta,\delta]}\|b\|_{[u-\delta,u+\delta]}$, so that the the first term converges to 0 as $h\rightarrow 0$.  As clearly the second term converges to $a(0)b(u)$, we conclude that
\begin{equation}
\lim_{h\rightarrow 0}\frac{1}{h}\int_{u}^{u+h}a(u+h-z)b(z)dz=a(0)b(u)\label{second}\end{equation}
Therefore, (\ref{princi})-(\ref{second}) show our result. \end{proof}
\begin{remark} The integrability condition for $b$ in Lemma \ref{letech} is authomatically satisfied if $b\in C([0,\infty))$\label{relete}\end{remark}

Next result gives conditions under which $m$ is differentiable, along with useful expressions for its derivatives.

\begin{proposition} Let $m$ be the locally bounded solution of \rmref{renewc} with $f$ the density function of an absolutely continuous nonnegative random variable $Z$ with distribution function $F$. Assume that $F$ and $v$ are in $C^{2}_{0}([0,\infty))$ We have that $m\in C^{2}_{0}([0,\infty))$ and for all $u>0$
\begin{align}
 m'(u)&=\phi\left(\int_{0}^{u}m(u-y)f'(y)dy+m(u)f(0)\right)+v'(u),\label{deri1f}\\ m''(u)&=\phi\left(\int_{0}^{u}m'(u-y)f'(y)dy+m(0)f'(u)+ m'(u)f(0)\right)+v''(u).\label{deri2f}\end{align}
Moreover,
\begin{eqnarray}
m(0)&=&v(0),\qquad m'(0)=\phi v(0)f(0)+v'(0),\quad\hbox{and} \label{deri10}\\
m''(0)&=&\phi(v(0)f'(0)+m'(0)f(0))+v''(0)\label{deri20}\end{eqnarray}
In addition, if we define the functions
\begin{eqnarray}
w_{1}(u)&:=&\phi m(0)f(u)+v'(u)\label{vstuno}\\
w_{2}(u)&:=&\phi m'(0)f(u)+w_{1}'(u).\label{vstdos}
\end{eqnarray}
We have
\begin{align}
 m'(u)&=\phi\int_{0}^{u}m'(u-y)f(y)dy+w_{1}(u) ,\label{deri1m}\\ m''(u)&=\phi\int_{0}^{u}m''(u-y)f(y)dy+w_{2}(u). \label{deri2m}\end{align}
\label{proderi}\end{proposition}
\begin{proof}First of all, from (\ref{renewc}), the fact that $m$ is locally bounded and the continuity of $v$ we deduce that $m$ is continuous on $[0,\infty)$.  Also, the first equality in (\ref{deri10}) follows by (\ref{rensol}.
 Secondly, making the change of variable $z=u-y$ in the integral contained in (\ref{renewc}) we can write
\[
m(u)=\phi\int_{0}^{u}m(z)f(u-z)dz+v(u),\quad u>0,\]
Differentiating the previous expression (recall Lemma \ref{letech} and Remark \ref{relete}), we obtain
\[
m'(u)=\phi\left(\int_{0}^{u}m(z)f'(u-z)dz+ m(u)f(0)\right)+v'(u),\quad u>0,\]
this shows that $m$ is continuoslly differentiable on $(0,\infty)$.
If we make now the change of variable $y=u-z$ in the previous expression we obtain (\ref{deri1f}).  Also, taking limits in this expression as $u\downarrow 0$, we obtain \[\lim_{u\downarrow 0}m'(u)=\phi m(0)f(0)+v'(0)=\phi v(0)f(0)+v'(0),\] thus showing that $m\in C_{0}^{1}([0,\infty))$ and the second equality in (\ref{deri10}). Now, (\ref{deri2f}) follows by differentiation in (\ref{deri1f}), and (\ref{deri20}) follows taking limits in (\ref{deri2f}) as $u\downarrow 0$

Finally once it is shown that $m\in C^{2}_{0}([0,\infty))$, Lemma \ref{letech} and Remark \ref{relete} allow us to make a straightforward differentiation in (\ref{renewc}) to obtain (\ref{deri1m}) and (\ref{deri2m}).
\end{proof}

In order to obtain bounds of the derivatives of $m$, the following technical result will be useful.

\begin{proposition} Let $m_{2}$ be a function satisfying
\begin{equation}
m_{2}(u)=\phi\int_{0}^{u}m_{1}(u-y)f_{1}(y)dy+v_{1}(u),\quad u>0, \label{renede}\end{equation}
in which $0<\phi<1$, $m_{1}$ and $v_{1}$ are continuous functions on $[0,\infty)$ and $f_{1}$ is a continuous function on $(0,\infty)$. Let
\begin{equation}
I_{i}(f_{1}):=\int_{0}^{\infty}y^{i}|f_{1}|(y)dy,\quad i=0,1,2,\dots\label{dfidef}\end{equation}
\begin{enumerate}[(a)]
\item If
$I_{0}(f_{1})<\infty$, then
\begin{equation}
\|m_{2}\|_{[0,x]}\leq \phi I_{0}(f_{1})\|m_{1}\|_{[0,x]}+\|v_{1}\|,\quad x>0, \label{reneb1}\end{equation}
\item If
$I_{i}(f_{1})<\infty,\ i=0,1$ we have, for all $x>0$
\begin{equation}
\kern-30pt\|um_{2}(u)\|_{[0,x]}\leq \phi \left( I_{0}(f_{1})\|um_{1}(u)\|_{[0,x]}+I_{1}(f_{1})\|m_{1}(u)\|_{[0,x]}\right)+\|uv_{1}(u)\|. \label{reneb2}\end{equation}
\item If
$I_{i}(f_{1})<\infty,\ i=0,1,2$ then, for all $x>0$
\begin{eqnarray}
\|u^{2}m_{2}(u)\|_{[0,x]}&\leq &\phi\left( I_{0}(f_{1})\|u^{2}m_{1}(u)\|_{[0,x]}+2I_{1}(f_{1})\|um_{1}(u)\|_{[0,x]}\right.\nonumber\\&&\left.+I_{2}(f_{1})\|m_{1}\|_{[0,x]}\right)+\|u^{2}v_{1}(u)\|.\label{reneb3}
\end{eqnarray}
\end{enumerate}\label{prcoun}
\end{proposition}

\begin{proof}
Part (a) is straightforward taking norms in (\ref{renede}).  To prove (b) and (c) we use that $u^{n}=(u-y+y)^{n}=\sum_{i=0}^{n}{n\choose i}y^{i}(u-y)^{n-i}$.  Therefore, using (\ref{renede}), we can write
\[u^{n}m_{2}(u)=\phi\sum_{i=0}^{n}{n\choose i}\int_{0}^{u}(u-y)^{n-i}m_{1}(u-y)y^{i}f_{1}(y)dy+u^{n}v_{1}(u)\]
from which we deduce easily
\[\|u^{n}m(u)\|_{[0,x]}\leq \phi\sum_{i=0}^{n}{n\choose i}\|u^{n-i}m_{1}(u)\|_{[0,x]}\int_{0}^{\infty}y^{i}|f_{1}|(y)dy+\|u^{n}v_{1}(u)\|\]Then, (b) and (c) follow easily from the previous expression applied to $n=1$ and $n=2$, respectively.\end{proof}

Taking into account the previous result and using the expressions (\ref{deri1m}) and (\ref{deri2m}), next result gives bounds for the weighted derivatives of $m$. Later on (in Proposition \ref{prcond}) we will give sufficient conditions in order to ensure the finiteness of these bounds.

\begin{proposition} Let $m$ be the locally bounded solution to \rmref{renewc} with $f$ the density function of an absolutely continuous nonnegative random variable $Z$, with distribution function $F$ and finite variance. Assume that  $F$ and $v$ are in $C^{2}_{0}([0,\infty))$. We have the following
\begin{enumerate}[(a)]

\item Let $w_{1}$ be as defined in \rmref{vstuno}. We have
\begin{eqnarray}\|m'\|&\leq &\frac{\|w_{1}\|}{1-\phi}\label{renod0}\\
\|u m'(u)\|&\leq &\frac{\phi EZ\|m'\|+\|u  w_{1}(u)\|}{1-\phi}\label{renod1}\\
\|u^{2}m'(u)\|&\leq &\frac{\phi(2EZ\|um'(u)\|+EZ^{2}\|m'\|)+\|u^{2} w_{1}(u)\|}{1-\phi}\label{renod2}
\end{eqnarray}
\item Let $w_{2}$ be as defined in \rmref{vstdos}. We have
\begin{eqnarray}\|m''\|&\leq &\frac{\|w_{2}\|}{1-\phi}\label{renos0}\\
\|u m''(u)\|&\leq &\frac{\phi EZ\|m''\|+\|u w_{2}(u)\|}{1-\phi}\label{renos1}\\
\|u^{2}m''(u)\|&\leq &\frac{\phi(2EZ\|um''(u)\|+EZ^{2}\|m''\|+\|u^{2} w_{2}(u)\|}{1-\phi}\label{renos2}
\end{eqnarray}

\end{enumerate}\label{corfbo}\end{proposition}
\begin{proof}To show the bounds in (a), we use (\ref{deri1m}) and apply Proposition \ref{prcoun}, with $m_{2}=m_{1}=m'$, $v_{1}=w_{1}$ and $f_{1}=f$ being a density function.  Thus, $I_{1}(f)=1$ and using (\ref{reneb1}), we can write
\[\|m'\|_{[0,x]}\leq \phi \|m'\|_{[0,x]}+\|w_{1}\|,\quad 0<x<\infty, \]
As by Proposition \ref{proderi} $m'$ is a continuous function on $[0,\infty)$, then $\|m'\|_{[0,x]}<\infty$, from which we deduce
\[(1-\phi)\|m'\|_{[0,x]}\leq \|w_{1}\|,\quad 0<x<\infty, \]
Taking limits as $x\rightarrow \infty$ in the previous expression, (\ref{renod0}) holds true.  (\ref{renod1}) and  (\ref{renod2}) are shown in a similar way, using (\ref{reneb2}) and (\ref{reneb3}), respectively, and taking into account that $I_{i}(f)=EZ^{i}, \ i=1,2$. The proof of part (b) is similar, taking into account (\ref{deri2m}) and applying Proposition \ref{prcoun}, with $m_{2}=m_{1}=m''$, $v_{1}=w_{2}$ and $f_{1}=f$\end{proof}

Our next task is to give conditions for $m$ in order to ensure that $m\in {\cal D}_{1}$. First of all, we state a technical lemma in order to simplify our hypothesis.
\begin{lemma}We have
\begin{enumerate}[(a)]
\item Let $v_{1}\in C([0,\infty))$.  If $\|u^{2}v_{1}(u)\|<\infty$, then $\|v_{1}\|<\infty$ and $\|uv_{1}(u)\|<\infty$.
\item Let $f_{1}\in C((0,\infty))$.  Let $I_{i}(f_{1})$ be as defined in \rmref{dfidef}.  If $I_{i}(f_{1})<\infty, \ i=0,2$, then $I_{1}(f_{1})<\infty.$
    \end{enumerate}\label{letec2}\end{lemma}
    \begin{proof}
 To show part (a), let $v_{1}\in C([0,\infty)$. The continuity implies $ \|v_{1}\|_{[0,1]}<\infty$ and the result is immediate using the the following bound
\[\|u^{i}v_{1}(u)\|\leq \|v_{1}\|_{[0,1]}+\|u^{2}v_{1}(u)\|_{(1,\infty)}<\infty\quad i=0,1\]

    Part (b) is due to Cauchy-Schwartz's inequality, as using that $u|f_{1}|(u)=u|f_{1}|^{1/2}(u)\cdotp |f_{1}|^{1/2}(u)$
\[I_{1}(f_{1})=\int_{0}^{\infty }u|f_{1}|(u)du\leq \left(\int_{0}^{\infty} u^{2}|f_{1}|(u)du\right)^{1/2}\left(\int_{0}^{\infty} |f_{1}|(u)du\right)^{1/2}= (I_{2}(f_{1})I_{0}(f_{1}))^{1/2}\]\end{proof}

Now we enunciate the main result of this section.
\begin{proposition} Let $m$  be the locally bounded solution to \rmref{renewc} with $f$ the density function of an absolutely continuous nonnegative random variable $Z$ with distribution function $F$. Assume that
\begin{enumerate}
\item $Z$ has finite variance.
\item $F$ and $v$ are in $C^{2}_{0}([0,\infty))$.
 \item $f'$ and $v''$ are in $C^{2}([0,\infty))$.
 \item $I_{i}(f^{''})<\infty$, $i=0,2$, where $I_{i}(\cdotp)$ is defined in \rmref{dfidef}.
 \item $\|u^{2}w_{i}(u)\|<\infty,\ i=1,2$, where  $w_{1}$, $w_{2}$, are defined in \rmref{vstuno} and \rmref{vstdos}, respectively.
 \item $\|u^{2}w''_{i}(u)\|<\infty,\ i=1,2$.
     \end{enumerate}
Then we have:

If Condition 2 is satisfied and in addition $\|w_{2}\|<\infty$, then,
\begin{equation}\|m''\|\leq \frac{\|w_{2}\|}{1-\phi}<\infty\label{bosec}\end{equation}

If Conditions 1-6 are satisfied, we have
\begin{align}\|u^{2}m'''(u)\|&\leq\phi \left((I_{0}(f'')+|f'(0)|)\|u^{2}m'(u)\|+2I_{1}(f'')\|um'(u)\|+I_{2}(f'')\|m'\|\right)\nonumber\\&+
\phi f(0)\|u^{2}m''(u)\|+\|u^{2}w''_{1}(u)\|<\infty \label{bothir}\\\|u^{2}m^{iv}(u)\|&\leq\phi \left((I_{0}(f'')+|f'(0)|)\|u^{2}m''(u)\|+2I_{1}(f'')\|um''(u)\|+I_{2}(f'')\|m''\|\right)
\nonumber\\&+\phi f(0)\|u^{2}m'''(u)\|+\|u^{2}w''_{2}(u)\|<\infty\label{bofour}\end{align}\label{prcond}
\end{proposition}

\begin{proof}Note firstly that (\ref{bosec}) is obvious by Condition 2 and Proposition \ref{corfbo} (b).
Secondly, note that Condition 2 allow us to apply Proposition \ref{proderi}.  Our starting point to prove (\ref{bothir})-(\ref{bofour})  will be the expression for $m''$ given in (\ref{deri2f}).  Moreover, as by condition 3 $f'$ is differentiable, the integral appearing in this expression can be rewritten, by means of an integration by parts as
\begin{equation}\int_{0}^{u}m'(u-y)f'(y)dy=\int_{0}^{u}m(u-y)f''(y)dy-m(0)f'(u)+m(u)f'(0)\label{auxddo}\end{equation}
 Inserting (\ref{auxddo}) in (\ref{deri2f}) , we obtain
\begin{equation}m''(u)=\phi\left(\int_{0}^{u}m(u-y)f''(y)dy+m(u)f'(0)+ m'(u)f(0)\right)+v''(u).\label{deri2a}\end{equation}
We differentiate the previous expression (note that by Proposition \ref{proderi} $m\in C_{0}^{2}([0,\infty)$ and, by Conditions 3 and 4, $b=f''$ satisfies conditions of Lemma \ref{letech}).  Thus, we obtain for all $u>0$
\begin{align}m'''(u)&=\phi\left(\int_{0}^{u}m'(u-y)f''(y)dy+m(0)f''(u)+\sum_{i=1}^{2}m^{(i)}(u)f^{(2-i)}(0) \right)+v'''(u)\nonumber\\&=\phi\left(\int_{0}^{u}m'(u-y)f''(y)dy+
\sum_{i=1}^{2}m^{(i)}(u)f^{(2-i)}(0)\right)+w_{1}''(u),\label{deri3f}\end{align}
where $w_{1}$ is defined in (\ref{vstuno}). Thus (\ref{deri3f}) verifies (\ref{renod0}) in Proposition
\ref{prcoun}, with \[m_{2}=m''',\ m_{1}=m',\ f_{1}=f'' \hbox{ and }v_{1}=\phi \sum_{i=1}^{2}m^{(i)}(u)f^{(2-i)}(0)+w_{1}''(u),\] so that we deduce from Proposition \ref{prcoun} (c)
\begin{align*}\|u^{2}m'''(u)\|&\leq\phi \left(I_{0}(f'')\|u^{2}m'(u)\|+2I_{1}(f'')\|um'(u)\|+I_{2}(f'')\|m'\|\right)\\&+
\phi\sum_{i=1}^{2}|f^{(2-i)}(0)|\|u^{2}m^{(i)}(u)\|+\|u^{2}w''_{1}(u)\|\end{align*}and the first inequality in (\ref{bothir}) follows  easily from the previous bound.  To show  (\ref{bofour}) we use Lemma \ref{letech} to differentiate (\ref{deri3f}), thus obtaining
\begin{align}m^{iv}(u)&=\phi\left(\int_{0}^{u}m''(u-y)f''(y)dy+
m'(0)f''(u)+\sum_{i=1}^{2}m^{(i+1)}(u)f^{(2-i)}(0)\right)+w_{1}'''(u)\nonumber \\&=\phi\left(\kern-3pt\int_{0}^{u}\kern-3pt m''(u-y)f''(y)dy+\sum_{i=1}^{2}m^{(i+1)}(u)f^{(2-i)}(0)\right)\kern-3pt+w_{2}''(u),\label{deri4f}\end{align}
where $w_{2}$ is defined in (\ref{vstdos}).  Then, as (\ref{deri4f}) verifies (\ref{renod0}) in Proposition
\ref{prcoun}, with
\[m_{2}=m^{iv},\ m_{1}=m'',\ f_{1}=f''\hbox{ and }v_{1}=\phi\sum_{i=1}^{2}m^{(i+1)}(u)f^{(2-i)}(0)+w_{2}''(u),\]
we obtain from Proposition \ref{prcoun} (c)
 \begin{align*}\|u^{2}m^{iv}(u)\|&\leq\phi \left(I_{0}(f'')\|u^{2}m''(u)\|+2I_{1}(f'')\|um''(u)\|+I_{2}(f'')\|m''\|\right)\\&+
\phi\sum_{i=1}^{2}|f^{(2-i)}(0)|\|u^{2}m^{(i+1)}(u)\|+\|u^{2}w''_{2}(u)\|,\end{align*} thus showing the first inequality in (\ref{bofour})  To show the finiteness of (\ref{bothir}) and (\ref{bofour}), we will prove that
\begin{eqnarray}\hbox{Conditions 1,2 and 5}&\Rightarrow& \|u^{j}w_{i}(u)\|<\infty,\ i=1,2,\ j=0,1,2\nonumber\\&\Rightarrow &\|u^{j}m^{(i)}(u)\|<\infty,\ i=1,2,\ j=0,1,2.\label{finfir}\end{eqnarray}  To show (\ref{finfir}), let $i=1,2$ be fixed. Condition 2 implies that  $w_{i}\in C[0,\infty)$.  Thus, by condition 5 and Lemma \ref{letec2} (a) we have the first implication in (\ref{finfir}). For the second implication we use condition 1 and apply Proposition \ref{corfbo} (a) for $i=1$, whereas for $i=2$ we apply Proposition \ref{corfbo} (b).

Now, note that

\begin{equation}\hbox{Condition 4}\Rightarrow  I_{i}(f'')<\infty,\quad i=1,2,3\label{finsec}\end{equation}
which is immediate by Lemma \ref{letec2} (b).
Thus, using (\ref{finfir}), (\ref{finsec}) and Condition 6 we show the finiteness of (\ref{bothir}). Similarly, the finiteness of the bound in  (\ref{bofour}) follows using (\ref{finfir}), (\ref{finsec}),  (\ref{bothir}) and Condition 6.  This completes the proof of Proposition \ref{prcond}
\end{proof}

As an immediate consequence of Proposition \ref{prcond} and Theorem \ref{tetota} we have the following.
\begin{corollary}Let $m$ be the locally bounded solution to \rmref{renewc} with $f$ the density function of an absolutely continuous nonnegative random variable $Z$ with distribution function $F$.  If conditions 1-6 in Proposition \rtref{prcond} are satisfied then $m\in {\cal D}_{1}$, with $ {\cal D}_{1}$ as defined in \rmref{dfcla2}. Therefore, the approximation $M_{t}^{[2]}m,\ t>0$, as defined in \rmref{dftile}-\rmref{dftilt} verifies \[\|M_{t}^{[2]}m-m\|\leq \frac{1}{8t^{2}}\|m''(u)\|
+\frac{1}{6t^{2}}\|um'''(u)\|+\frac{9}{16t^{2}}\|u^{2}m^{iv}(u)\|<\infty.\]\label{cofina}\end{corollary}

\section{Application: Approximations for ruin probabilities}
In this section we will apply the results given in the previous one to ruin probabilities in the classical risk model, which are a well-known example for functions satisfying a defective renewal equation.  First of all we recall how the classical risk model is defined (see \cite[Ch. 4]{kagomo} or \cite[Ch.5.3]{roscst}, for instance).  We consider an insurance company in which insurance claims follow a Poisson process $(N(t), \ t\geq 0)$  with intensity $\lambda>0$. On the other hand, the individual claim
amounts $\seq{X}$ are identically distributed and positive random
variables with finite mean, independent on $(N(t), \ t\geq 0)$ .
Suppose that the initial capital of the insurance company is
$U(0):=u\geq 0$ and it receives premiums at a constant rate $c$. In this setting, the
probability of eventual ruin $\psi(u)$ is the probability that the
wealth of the company is ever negative, i.e.
\begin{equation} \psi(u)=P\left(\inf_{t\geq0}\left(u+ct-\sum_{i=0}^{N(t)}X_{i}\right)<0\right),\quad u\geq 0, \label{ruinpr}\end{equation}
Call $\mu:=EX_{1}$, and assume that $\mu>0$.  The condition for no sure ruin is
 \begin{equation}
\phi:=\frac{\lambda \mu}{c}<1\label{safelo}\end{equation}
We will assume this condition from now on.  Usually the ruin function cannot be evaluated in an explicit way. However, it is well-known that the ruin function satisfies a defective renewal equation (cf. \cite[p. 105]{kagomo}, for instance).  In fact, let $F_{X}$ be the distribution function of the claim amounts, and $\bar{F}_{X}:=1-F_{X}$. We have
\begin{equation}
\psi(u)=\phi\left(\int_{0}^{u}\psi(u-y)\frac{\bar{F}_{X}(y)}{\mu}dy+\int_{u}^{\infty}\frac{\bar{F}_{X}(y)}{\mu}dy\right),\quad u\geq0. \label{renewr}\end{equation}
and therefore (\ref{renewc}) holds true for the ruin probability, with
\begin{equation}F'(u)=f(u)=\frac{\bar{F}_{X}(u)}{\mu}\quad \hbox{and} \quad v(u):=\phi\int_{u}^{\infty}f(y)dy=\phi(1-F(u))\label{efeuve}\end{equation}
Note that $f$ is a well-defined density corresponding to the so-called equilibrium distribution of $X$ (cf. \cite[p. 14]{wililu}).  The rest of the section is divided in two parts.  In the first one we will check that if the claim amounts are mixtures of gamma random variables with shape parameter $\alpha\geq 1$ (and arbitrary scale parameter), then $\Psi\in {\cal D}_{1}$, so that Corollary \ref{cofina} holds true. In the second part, we will give a method to compute $M_{t}^{[2]}\Psi$, paying special attention to mixtures of gamma claim amounts.
\subsection{Conditions for optimal order of convergence in the approximated ruin probability.  Applications to mixtures of gamma claim amounts}
Our first result gives sufficient conditions in order that $\Psi\in {\cal D}_{1}$.  As an immediate consequence we will obtain sufficient conditions for mixtures of gamma claim amounts.
\begin{proposition} Consider the classical risk model, where the claim amounts $X$ have distribution $F_{X}$.  Assume that $F_{X}\in C^{3}[0,\infty)$ and
\begin{enumerate}[a)]
\item $X$ has finite third moment
\item $F'_{X}$ verifies that $\displaystyle \lim_{u\downarrow 0}F_{X}'(u)$ exists and is finite.
\item $I_{i}(F''_{X})=\int_{0}^{\infty}u^{i}|F_{X}''(u)|du<\infty,\ i=0,2$
\item $\|u^{2}\bar{F}_{X}(u)\|<\infty,$ and $\|u^{2}F^{(i)}_{X}(u)\|<\infty,\ i=1,2,3$
\end{enumerate}  Then, the associated functions $F$ and $v$ for the ruin probability $\Psi$, as given in \rmref{efeuve} satisfy conditions 1-6 in Proposition \rtref{prcond}\label{prruge}\end{proposition}

\begin{proof} We will consider the renewal equation for the ruin function as given in (\ref{renewr}) and (\ref{efeuve}) and check all the conditions in Proposition \ref{prcond}.

\noindent {\it Condition 1}.  Let $Z$ be the random variable whose distribution is $F$, as given in (\ref{efeuve}). To show that $Z$ has finite variance, we note that an integration by parts shows us that $EZ^{2}=EX^{3}/(3EX)$ (cf \cite[p.15]{wililu}) so, that if a) is satisfied, Condition 1 holds true.

To check the rest of conditions, taking into account (\ref{efeuve}), the functions for the renewal equation and their respective derivatives are
\begin{align} F'(u)&=f(u)=\frac{\bar{F}_{X}(u)}{\mu},\quad \hbox{and}\quad F''(u)=f'(u)=\frac{-F_{X}'(u)}{\mu}\nonumber\\ v(u)&=\phi(1- F(u)),\quad v'(u)=\phi\frac{-\bar{F}_{X}(u)}{\mu}\quad \hbox{and}\quad v''(u)=\phi \frac{F_{X}'(u)}{\mu}\label{sedeFv}\end{align}

\noindent {\it Condition  2}. Taking into account (\ref{sedeFv}) it is immediate that if b) holds true, then $F$ and $v$ are in $C^{2}_{0}([0,\infty))$.

\noindent {\it Condition  3}. Taking into account (\ref{sedeFv}), both $F''$ and $v''$ are in $C^{2}[0,\infty)$, as $F_{X}'\in C^{2}[0,\infty)$.

\noindent {\it Condition  4}. We need to chack that $I_{i}(f'')<\infty$, $i=0,2$.  This follows from c), as $f''=-F_{X}''/\mu$.

To show Conditions 5-6 we recall (\ref{vstuno}) and use (\ref{sedeFv}) and (\ref{deri10}) to write
\begin{equation}w_{1}(u)=\phi v(0)f(u)+v'(u)=-\phi(1-\phi) \frac{\bar{F}_{X}(u)}{\mu}\label{vstung}\end{equation}
Now we recall (\ref{vstdos}) and use (\ref{sedeFv}) and (\ref{vstung}) to write
\begin{equation}w_{2}(u)=\phi \psi'(0)f(u)+w_{1}'(u)=\phi \psi'(0)\frac{\bar{F}_{X}(u)}{\mu}+\phi(1-\phi)\frac{F_{X}'(u)}{\mu} \label{vstdo1}\end{equation}
Recalling (\ref{deri10}) and using (\ref{sedeFv}) we have that
\[\psi'(0)=\phi f(0) v(0)+v'(0)=-\frac{\phi (1-\phi)\bar{F}_{X}(0)}{\mu}\]  and using the previous equality, (\ref{vstdo1}) and (\ref{vstung}) we can write
\begin{equation}w_{2}(u)=\frac{\phi \bar{F}_{X}(0)}{\mu} w_{1}(u)+\phi(1-\phi) \frac{F_{X}'(u)}{\mu}.\label{vstdog}\end{equation}
{\it Condition  5}. We need to show that  $\|u^{2}w_{i}(u)\|<\infty, i=1,2$. Taking into account d), we have that $\|u^{2}\bar{F}_{X}(u)\|<\infty,$ which implies, recalling (\ref{vstung}), that $\|u^{2}w_{1}(u)\|<\infty$.  Using the previous bound, (\ref{vstdog}) and the fact that, by d), $\|u^{2}F_{X}'(u)\|<\infty$, we have that $\|u^{2}w_{2}(u)\|<\infty$

\noindent {\it Condition  6}. We need to show that  $\|u^{2}w''_{i}(u)\|<\infty, \ i=1,2$. This follows easily by (\ref{vstung}), (\ref{vstdog}) and d), as
\begin{align*}\|u^{2}w_{1}''(u)\|&=\frac{\phi(1-\phi)}{\mu}\| u^2F_{X}''(u)\|<\infty \\
\|u^{2}w_{2}''(u)\|&\leq \frac{\phi \bar{F}_{X}(0)}{\mu}\|u^{2}w_{1}''(u)\|+\frac{\phi(1-\phi)}{\mu}\|u^{2}F_{X}'''(u)\|<\infty
\end{align*}
This completes the proof of Proposition \ref{prcond}\end{proof}
\begin{corollary} Assume that, in the classical risk model, the claim amounts are mixtures
 of gamma random variables $\Gamma(\alpha_{i},\beta_{i})$, with mixing weights $\seqn{p}$, that is
\begin{equation}F_{X}(u)=\sum_{i=1}^{n}p_{i}F_{\alpha_{i},\beta_{i}}(u),\label{gammix}\end{equation}
where $F_{\alpha_{i},\beta_{i}}$ are distributions having density  $f_{\alpha_{i},\beta_{i}}$ as defined in \rmref{gamden}, $p_{i}> 0$ and $p_{1}+\dots+p_{n}=1$.  Assume that
$\alpha_{i}\geq 1, \ i=1,\dots, n$.  Then, the ruin probability $\psi$ satisfies that $\psi\in {\cal D}_{1}$, with $ {\cal D}_{1}$ as defined in \rmref{dfcla2}. Therefore, the approximation $M_{t}^{[2]}\psi, \ t>0$, as defined in \rmref{dftile}-\rmref{dftilt} has uniform order of convergence $1/t^{2}$, that is
  \[\|M_{t}^{[2]}\psi-\psi\|\leq \frac{1}{8t^{2}}\|\psi''(u)\|
+\frac{1}{6t^{2}}\|u\psi'''(u)\|+\frac{9}{16t^{2}}\|u^{2}\psi^{iv}(u)\|<\infty.\] \label{coruge}\end{corollary}
\begin{proof} To prove the result we will show that $F_{X}$ satisfies the conditions of Proposition \ref{prruge}. Condition a) is obviously satified.  To check the rest of conditions, we will use the following simplification.  Denote by $F_{\alpha}:=F_{\alpha,1}$, a gamma distribution  $\Gamma(\alpha,\beta=1)$. Recall that a gamma distribution $\Gamma(\alpha,\beta)$, has distribution function $F_{\alpha,\beta}(u)=F_{\alpha}(\beta u)$. Therefore, we can write (\ref{gammix}) as
\[F_{X}(u)=\sum_{i=1}^{n}p_{i}F_{\alpha_{i}}(\beta_{i}u),\quad\hbox{and}\quad \bar{F}_{X}(u)=\sum_{i=1}^{n}p_{i}\bar{F}_{\alpha_{i}}(\beta_{i}u) \] and it is clear that if $F_{\alpha},\ \alpha\geq 1$ satisfies conditions b)-d) in Proposition \ref{prruge}, then $F_{X}$ will also do. Then, we will check conditions b)-d) for $F_{\alpha}$, with $\alpha\geq 1$.  Note firstly that  $F_{\alpha}$ is infinitely differentiable on $(0,\infty)$.  In particular, its density
\begin{equation} F_{\alpha}'(u)=\frac{1}{\Gamma(\alpha)}e^{-u}u^{\alpha-1},\quad u\geq 0.\label{gamde1}\end{equation}
satisfies condition b) in Proposition \ref{prruge}, whenever $\alpha\geq 1$. For the rest of conditions, note that
 \begin{align}F_{\alpha}''(u)&=\frac{1}{\Gamma(\alpha)}e^{-u}u^{\alpha-2}(\alpha-1-u),\quad u>0,\label{gamdem}\\F_{\alpha}'''(u)&=\frac{1}{\Gamma(\alpha)}e^{-u}u^{\alpha-3}((\alpha-1)(\alpha-2)-2(\alpha-1)u+u^{2}),\quad
u>0.\label{gamde2}\end{align}
Observe that the previous equalities follow by differentiation in (\ref{gamde1}) for $\alpha> 1$, and are still valid for $\alpha=1$.
To show condition c) in Proposition \ref{prruge}, note that
\begin{align*}I_{i}(F_{\alpha}'')&=\int_{0}^{\infty}u^{i}|F_{\alpha}''|(u)du\leq \frac{(\alpha-1)}{\Gamma(\alpha)}\int_{0}^{\infty}u^{i}e^{-u}u^{\alpha-2}+\frac{1}{\Gamma(\alpha)}
\int_{0}^{\infty}u^{i}e^{-u}u^{\alpha-1}\\&=\frac{(\alpha-1)\Gamma(\alpha-1+i)}{\Gamma(\alpha)}+\frac{\Gamma(\alpha+i)}{\Gamma(\alpha)}<\infty ,\quad i=0,2, \quad \alpha\geq1.\end{align*}
To show Condition d), note that it is clear from (\ref{gamde1})-(\ref{gamde2}) that
\begin{equation}\lim_{u\rightarrow \infty}u^{2}F_{\alpha}^{(i)}(u)=0,\quad i=1,2,3\label{liminf}\end{equation}
and using L'Hopital'rule applied to $u^{2}\bar{F}_{\alpha}(u)$, we can write
\begin{equation}\lim_{u\rightarrow \infty}u^{2}\bar{F}_{\alpha}(u)=\lim_{u\rightarrow \infty}2 u^{3}F_{\alpha}'(u)=0\label{limin0}\end{equation}
From (\ref{liminf}), (\ref{limin0})  and the fact that $\bar{F}_{\alpha}$ and $F_{\alpha}'$ are continuous and bounded at the origin we obtain  that $\|u^{2}\bar{F}_{\alpha}(u)\|<\infty$ and $\|u^{2}F_{\alpha}'(u)\|<\infty$.  To show that  $\|u^{2}F_{\alpha}^{(i)}(u)\|<\infty,\ i=2,3$, we take into account (\ref{liminf}) and note that
\begin{align*}\lim_{u\rightarrow 0}u^2F''_{\alpha}(u)&=0,\quad \alpha\geq 1,\\\lim_{u\rightarrow 0}u^2F'''_{\alpha}(u), &=\frac{(\alpha-1)(\alpha-2)}{\Gamma(\alpha)}\lim_{u\rightarrow 0}u^{\alpha-1}=0, \quad \alpha\geq 1.\end{align*}
Note that the last equality follows as, for $\alpha=1$, $\alpha-1=0$, whereas for $\alpha>1$, $\lim_{u\rightarrow 0}u^{\alpha-1}=0$.
Thus conditions a)-d) in Proposition \ref{prcond} are verified, and the conclusion follows by Corollary \ref{cofina} \end{proof}

\subsection{Numerical computation of the approximated ruin probability.  Applications to mixtures of gamma claim amounts }
 In this Section we will give a method to compute the approximated ruin probability, using that non-ruin probability is the distribution function of a geometric sum.  Our approach is based on the following representation of $L_{t}^{*}g$, as defined in (\ref{dflapt}), when $g:=F_{X}$ is the distribution function of a nonnegative random variable $X$. In this case, the approximation $L_{t}^{*}F_{X}$ can be rewritten in the following terms (cf. \cite{adcaap}).  Let
$\phi_{X}(\cdotp)$ be the Laplace-Stieltjes transform of $X$, that
is
\[\Phi_{X}(t):=Ee^{-tX}=\int_{[0,\infty)}e^{-t u}dF_{X}(u),\quad t>0.\]
We define a random
variable $\discre{X}{t}$ taking values on $k/t,\quad k\in \Nn$, and
such that
\begin{equation}P(\discre{X}{t}=k/t)=\frac{(-t)^{k}}{k!}\Phi_{X}^{(k)}(t),\quad k\in
\Nn,\label{discre}\end{equation}
Let $(S(t),\ t\geq 0)$ be a collection of gamma random variables $\Gamma(t,1)$, as considered in the Introduction.  The following equality holds true (see \cite{adcaap}),
\begin{equation}P(\discre{X}{t}\leq
u)=\sum_{k=0}^{[tu]}\frac{(-t)^{k}}{k!}\phi^{(k)}_{X}(t)=EF_{X}\left(\frac{S([tu]+1)}{t}\right)
,\quad u\geq
0,\label{dfdfdi}\end{equation}
 As mentioned in the Introduction, the first equality is the so-called Widder's formula in \cite{matenu}. Let $L_{t}^{*}F_{X}$ be the approximation defined in (\ref{delets}).  Recalling (\ref{repgam}) we see therefore that
\begin{equation}L_{t}^{*}F_{X}(u)=\frac{(-t)^{[tu]+1} }{ \Gamma([tu]+1)}\widetilde{F_{X}}^{([tu])}(t)=P(\discre{X}{t}\leq
u)\label{discret}\end{equation}
Thus, $L_{t}^{*}F$ can be obtained either by straightforward differentiation using the first equality or by computing the probability mass function of $\discre{X}{t}$ (second equality).  Now, we will see the computational advantages of using the second method to approximate $\psi$, the ruin probability in the classical risk model, as defined in Section 3.1. To this end, denote by $\bar{\psi}$ the non-ruin probability, that is $\bar{\psi}=1-\psi$. It is well-known (cf. \cite[p.104]{kagomo}, for instance) that
\begin{equation}\bar{\psi}(u)=P(\sum_{i=1}^{M}L_{i}\leq u)\label{ruigeo}\end{equation}
in which $M$ is a geometric random variable, with probability of 'succes' $p=1-\phi$, where $\phi$ is as in  (\ref{safelo}).  That is, $P(M=n)=p(1-p)^{n}=(1-\phi)\phi^{n},\quad n=0,1,\dots$, and $\seq{L}$ is a sequence of i.i.d. random variables, having density $f$ as defined in (\ref{efeuve}), and independent of $M$. Applying (\ref{discret}) to non-ruin probability and taking into account that $\displaystyle \discre{\left(\sum_{i=1}^{M}L_{i}\right)}{t}$ has the same distribution as $\sum_{i=1}^{M}\discre{L_{i}}{t}$ (see \cite[Proposition 2.1.]{saerro}) we can write
\begin{equation}L_{t}^{*}\bar{\psi}(u)=P\left(\discre{\left(\sum_{i=1}^{M}L_{i}\right)}{t}\leq
u\right)=P\left(\sum_{i=1}^{M}\discre{L_{i}}{t}\leq
u\right),\quad u\geq 0\label{discrer}\end{equation}and using the accelerated approximation $M_{t}^{[2]}\bar{\psi}$ as defined in (\ref{dftile}) we can write
\[M_{t}^{[2]}\bar{\psi}\left(\frac{k}{t}\right)=2P\left(\sum_{i=1}^{M}\discre{L_{i}}{2t}\leq \frac{2k-1}{2t}\right)-P\left(\sum_{i=1}^{M}\discre{L_{i}}{t}\leq \frac{k-1}{t}\right),\ k=1,2,\dots \]
Therefore, we can approximate the ruin probability by evaluating the distribution function of a discrete compound geometric distribution.  This allows us to use well-known evaluating techniques for compound discrete distributions (Panjer's recursion, for instance, see \cite[p.50]{kagomo})).  Approximations for ruin probabilities by means of the discretization of the summands in (\ref{ruigeo}) have been proposed in the literature (cf. \cite[p.110]{kagomo}).  Perhaps the most natural way to discretize a random variable is to round it from below or from above.  However, rounding methods are difficult to apply when the distribution function of a random variable cannot be given in an explicit way (consider a gamma random variable with a shape parameter being not a natural number, or its equilibrium distribution, for instance).  The computational advantage of our method is that we can evaluate the probability mass function of $\discre{L_{i}}{t}$, whenever the Laplace-Stieltjes transform of the claim amounts is known. The expression for the discretized record lows and their behaviour when dealing with mixtures, are collected in the following.

\begin{proposition} Consider a non-negative random variable $X$ with distribution function $F_{X}$ and Laplace-Stieltjes transform $\Phi_{X}$. Assume that $X$ has finite mean $\mu$.  Let $L$ be a random variable having the equilibrium distribution of $X$, that is whose density is given as
\begin{equation}f_{L}(u)=\frac{\bar{F}_{X}(u)}{\mu},\quad u\geq 0\label{efeuv2}\end{equation}
 We have the following.
 \begin{enumerate}[(a)]
 \item Let $\discre{L}{t}$ the discretization given in \rmref{discre}.  We have
 \[\displaystyle P(\discre{L}{t}=\frac{k}{t})=\frac{1}{t\mu}\left(1-\sum_{j=0}^{k}\frac{(-1)^{j}t^{j}}{j!}\Phi_{X}^{(j)}(t)\right)\]
 \item Assume that $F_{X}$ is a mixure of random variables
\[F_{X}=p_{1}F_{1}+\dots +p_{n}F_{n}\]
where the mixing distribution functions  $\seqn{F}$, have finite mean $\seqn{\mu}$ and $\seqn{p}$ are the mixing probabilities.

Let $\seqn{L}$ be random variables having the equilibrium distribution of $\seqn{F}$. Then $\discre{L}{t}$ as given in \rmref{discre} verifies
\begin{equation}P(\discre{L}{t}=\frac{k}{t})=\frac{p_{1}\mu_{1}}{\mu}P(\discre{L_{1}}{t}=\frac{k}{t})+\dots+\frac{p_{n}\mu_{n}}{\mu}P(\discre{L_{n}}{t}=\frac{k}{t})\label{apeqmi}
\end{equation}
\end{enumerate}\label{preval}\end{proposition}
\begin{proof} To show (a), we use an integration by parts to write
\[\Phi_{L}(t)=\int_{[0,\infty)}e^{-t u}\frac{\bar{F}_{X}(u)}{\mu}du=\frac{1}{t\mu}\left(1-\Phi_{X}(t)\right)\]
and therefore, applying Leibnitz's differentiation rule, we can write
\[\Phi_{L}^{(k)}(t)=\frac{1}{\mu}\left(\frac{(-1)^{k}k!}{t^{k+1}}-\sum_{j=0}^{k}{k\choose j}\frac{(-1)^{k-j}(k-j)!}{t^{k-j+1}}\Phi_{X}^{(j)}(t)\right)
\]
Thus, using (\ref{discre}), we can write
\[P(\discre{L}{t}=\frac{k}{t})=\frac{(-t)^{k}}{k!}\Phi_{L}^{(k)}(t)=\frac{1}{t\mu}\left(1-\sum_{j=0}^{k} \frac{(-1)^{j}t^{j}}{j! }\Phi_{X}^{(j)}(t)\right)\]
thus proving (a).
To show (b), note that we can write
\[\frac{\bar{F}_{X}}{\mu}=\frac{p_{1}\mu_{1}}{\mu}\frac{\bar{F}_{1}}{\mu_{1}}+\dots +\frac{p_{1}\mu_{1}}{\mu}\frac{\bar{F}_{n}}{\mu_{n}}\]
and, taking into account the previous expression we have
\[\Phi_{L}(t)=\int_{[0,\infty)}e^{-t u}\frac{\bar{F}_{X}(u)}{\mu}du=\frac{p_{1}\mu_{1}}{\mu}\Phi_{L_{1}}(t)+\dots +\frac{p_{n}\mu_{n}}{\mu}\Phi_{L_{n}}(t)\]
Thus,
\begin{eqnarray*}P(\discre{L}{t}=\frac{k}{t})&=&\frac{(-t)^{k}}{k!}\Phi_{L}^{(k)}(t)=\frac{(-t)^{k}}{k!}\left(\frac{p_{1}\mu_{1}}{\mu}\Phi_{L_{1}}^{(k)}(t)+\dots +\frac{p_{n}\mu_{n}}{\mu}\Phi_{L_{n}}^{(k)}(t)\right)\\&=&\frac{p_{1}\mu_{1}}{\mu}
P(\discre{L_{1}}{t}=\frac{k}{t})+\dots+\frac{p_{n}\mu_{n}}{\mu}P(\discre{L_{n}}{t}=\frac{k}{t})\end{eqnarray*}
thus showing part (b)\end{proof}

As an immediate application of the previous result, to mixtures of gamma random variables, we have the following.
\begin{corollary}  Consider $X$ a nonnegative random variable, let $L$ be its equilibrium distribution as given in \rmref{efeuv2} and consider its discretization $\discre{L}{t}$ as given in \rmref{discre}.  We have the following
\begin{enumerate}[(a)]
\item Assume that $X$ is a gamma random variable $\Gamma(\alpha,\beta)$, that is, having density as given in \rmref{gamden}.  Consider the cumulative distribution of a negative binomial random variable with $\alpha>0$ 'successes' and probability of 'success' $\rho$, that is
\[CDF.NB(k;\alpha,\rho)=\sum_{j=0}^{k}{\alpha+j-1\choose j}(1-\rho)^j \rho^{\alpha},\quad k=0,1,\dots\]
Then,
\begin{equation}P(\discre{L}{t}=\frac{k}{t})=\frac{\beta}{t\alpha}\left(1-CDF.NB\left(k;\alpha,\rho=\frac{\beta}{t+\beta}\right)\right)\end{equation}

\item Assume that $X$ is a mixture of $n$ gamma random variables, with mixing weights $\seqn{p}$, that is
\[F_{X}=p_{1}F_{1}+\dots +p_{n}F_{n}\]
in which each $F_{i}$ has distribution $\Gamma(\alpha_{i},\beta_{i})$. Then,
\begin{equation}P(\discre{L}{t}=\frac{k}{t})=
\frac{\displaystyle \sum_{i=1}^{n}p_{i}\left(1-CDF.NB\left(k;\alpha_{i},\frac{\beta_{i}}{t+\beta_{i}}\right)\right)}{\displaystyle t\left(\frac{p_{1}\alpha_{1}}{\beta_{1}}+\dots+\frac{p_{n}\alpha_{n}}{\beta_{n}}\right)}\label{apeqm1}\end{equation}
\end{enumerate}\label{corlmg}
\end{corollary}
\begin{proof}
To show (a) note that the Laplace-Stieltjes transform of a $\Gamma(\alpha,\beta)$ random variable is
\[\Phi_{X}(t)=\left(\frac{\beta}{t+\beta}\right)^{\alpha}\]
and therefore,
\[\Phi_{X}^{(j)}(t)=(-1)^{j}\frac{\Gamma(\alpha+j)}{\Gamma(\alpha)}\frac{\beta^{\alpha}}{(t+\beta)^{\alpha+j}}=(-1)^{j}j!{\alpha+j-1\choose j}\frac{\beta^{\alpha}}{(t+\beta)^{\alpha+j}}\]
Recalling that $\mu=\alpha/\beta$ and applying Proposition \ref{preval} (a), we have
\begin{equation}P(\discre{L}{t}=\frac{k}{t})=\frac{\beta}{t\alpha}\left(1-\sum_{j=0}^{k}{\alpha+j-1\choose j}\left(\frac{t}{t+\beta}\right)^j \kern -2pt\left(\frac{\beta }{t+\beta
}\right)^{\alpha}\right),\label{apgaeq}\end{equation}
which shows (a).  Part (b) is immediate by part (a) and Proposition \ref{preval} (b), taking into account that $\mu_{i}=\alpha_{i}/\beta_{i}$.
\end{proof}
\begin{remark} As shown in {\rm \cite{adcaap}}, when $X$ is a gamma random variable $\Gamma(\alpha,\beta)$, the weights of $\discre{X}{t}$, as defined in
\rmref{discre} correspond to the ones of a negative binomial random variable. From Corollary \rtref{corlmg} (a) we deduce that the discretized equilibrium distribution of a gamma $\Gamma(\alpha,\beta)$ random variable is constructed by cumulative sum of the afore-mentioned weights. \end{remark}

\noindent {\bf Example 3.1.  Approximation of ruin probabilities when the claim amounts are mixtures of gamma distributions}
In this example we show some numerical approximations of the ruin probabilites with the method described above, by considering mixtures of gamma claim amounts.
 First of all we describe the steps needed to build the approximation
  \begin{enumerate}[1.]
\item Computation of $P(\discre{L}{t}=\frac{k}{t})$ by the mixture formula given in (\ref{apeqm1}).  For fixed $t$ and fixed values of $\seqn{p}$, $\seqn{\alpha}$ and $\seqn{\beta}$, we need the probability distribution of the corresponding negative binomials.  In our case, we used MATLAB to generate these values, for $k=0,1,2,\dots$

\item Computation of $L_{t}^{*}\bar{\psi}$, using (\ref{discrer}).  Note that this can be done using Panjer's recursion, a popular method for evaluating compound distributions. Panjer's recursion applied to the geometric sum given in (\ref{discrer}) provides the following recursive formula for evaluating the probability mass function of $\displaystyle \sum_{i=1}^{M}\discre{L_{i}}{t}$ (cf. \cite[p. 50]{kagomo})
\[P\left(\sum_{i=1}^{M}\discre{L_{i}}{t}=\frac{k}{t}\right)=\left\{
                                                    \begin{array}{lll}
                                                      \displaystyle \frac{1-\phi}{1-\phi P(\discre{L_{i}}{t}=0)}, & \hbox{if $k=0$;} \\&\\
                                                    \displaystyle  \frac{\displaystyle \phi  \left(\sum_{j=1}^{k}P\left(\discre{L}{t}=\frac{j}{t}\right)P\left(\sum_{i=1}^{M}\discre{L_{i}}{t}=\frac{k-j}{t}
\right)\right)}{1-\phi P(\discre{L}{t}=0)} , & \hbox{if $k=1,2,\dots$.}
                                                    \end{array}
                                                  \right.\]
In our case we used an EXCEL worksheet to generate, for fixed $t$, and the computations given in the previous step, the corresponding values of $L_{t}^{*}\bar{\psi}(k/t)$.
\item The final approximation for the non-ruin probability is
\[M_{t}^{[2]}\bar{\psi}\left(\frac{k}{t}\right)=2P\left(\sum_{i=1}^{M}\discre{L_{i}}{2t}\leq\frac{2k-1}{2t}\right)-P\left(\sum_{i=1}^{M}\discre{L_{i}}{t}\leq \frac{k-1}{t}\right)\]
Note that this implies to repeat steps 1 and 2 for $t$ (second term above) and $2t$ (first term above)
\end{enumerate}
If the shape parameters  $\seqn{\alpha}$ of the initial mixture of claim amounts are chosen to be $\alpha_{i}\geq 1, \ i=1,\dots,n$, Corollary \ref{coruge} ensures us a uniform order of convergence of $1/t^{2}$ .

 We give three numerical computational examples.  For them we use $\phi=0.9$ and $t=5$. Results are shown in Table \ref{table1}
\begin{itemize}
\item First column provides us the approximation of non-ruin probability for exponential claim amounts, having mean 1 that is, $\Gamma(\alpha=1,\beta=1)$.  This example can be used as a test, as exact non-ruin probabilities can be computed in an exact way by the formula (cf. \cite[p.93]{kagomo})
    \[\bar{\psi}(u)=1-(1-p) e^{-p u},\quad u\geq 0,\]
    where $p=1-\phi$, as above.  Note that this function was the used in \cite[Example 2.1] {saunif} for numerical computations (recall Remark \ref{renume}), and provides, with the given parameters, exact values up to four decimal places. With respect to Column 1 in Table 1 the approximation is recalculated using the Panjer's recursion method described above, whereas in \cite{saunif} the explicit expression for $M_{5}^{[2]}\bar{\psi}$ was used.
\item Second column provides us non-ruin probability for Gamma distributed claim amounts $\displaystyle \Gamma(\alpha=\frac{3}{2},1)$. The interest of using the approximation in this case is that, when $\alpha\not \in \Nn$, there is no explicit expression for the ruin probability. However, alternative approximate expressions can be obtained by series expansions (see \cite{willmot}).
\item Third column provides us non-ruin probability for claim amounts being a mixture of the previous cases, with weights $p_{1}=p_{2}=1/2$.  We have chosen the same scale parameter in both terms, to easen comparability, but note that there is no computational problem in choosing different scale parameters.
\end{itemize}
 \begin{table}[h]
  \centering
  \begin{tabular}{|c|l|l|l|l|l|}
    \hline
     $\displaystyle u=\frac{k}{5}$ &Exponential claims&Gamma $\alpha=3/2$ claims&Mixture\\\hline
 1=$\displaystyle \frac{5}{5}$&  $M_{5}^{[2]}\bar{\psi}(1) = 0.1856$ & $M_{5}^{[2]}\bar{\psi}(1) =0.1648$& $M_{5}^{[2]}\bar{\psi}(1) =0.1726$

\\
 5=$\displaystyle \frac{25}{5}$&   $M_{5}^{[2]}\bar{\psi}(5) = 0.4538$& $M_{5}^{[2]}\bar{\psi}(5)=0.3940$&$M_{5}^{[2]}\bar{\psi}(5) =0.4159$

 \\
 10=$\displaystyle \frac{50}{5}$&  $M_{5}^{[2]}\bar{\psi}(10) =   0.6677$& $M_{5}^{[2]}\bar{\psi}(10) = 0.5949$& $M_{5}^{[2]}\bar{\psi}(10) = 0.6225$

\\
  15=$\displaystyle \frac{75}{5}$&   $M_{5}^{[2]}\bar{\psi}(15) =   0.7975$&  $M_{5}^{[2]}\bar{\psi}(15) =0.7248$&$M_{5}^{[2]}\bar{\psi}(15) = 0.7560$

\\
20=$\displaystyle \frac{100}{5}$&  $M_{5}^{[2]}\bar{\psi}(20) =   0.8766$& $M_{5}^{[2]}\bar{\psi}(20)=0.8190$& $M_{5}^{[2]}\bar{\psi}(20) = 0.8423$

\\
  30=$\displaystyle \frac{150}{5}$&  $M_{5}^{[2]}\bar{\psi}(30) =    0.9553$& $M_{5}^{[2]}\bar{\psi}(30) =0.9191$& $M_{5}^{[2]}\bar{\psi}(30) =0.9341$

\\
  40=$\displaystyle \frac{200}{5}$&  $M_{5}^{[2]}\bar{\psi}(40) = 0.9854$&$M_{5}^{[2]}\bar{\psi}(40) =0.9639$&$M_{5}^{[2]}\bar{\psi}(40) =0.9725$
\\
    \hline
  \end{tabular}
  \caption{Approximation of non-ruin probability for different claim amounts} \label{table1}
\end{table}

\section*{Acknowledgments}
 This research has been supported by the research grants MTM2010-15311, E64 (DGA) and by FEDER funds.

\end{document}